\documentclass[english,a4paper,12pt]{amsart}

\usepackage[a4paper,lmargin=2cm,rmargin=2cm,tmargin=4cm,bmargin=4cm]{geometry}
\usepackage[centertags]{amsmath}
\usepackage{amsfonts}
\usepackage{amssymb}
\usepackage{amsthm}
\usepackage[colorlinks]{hyperref}

\newcommand{\Real}{\mathbb R}

\newcommand{\abs}[1]{\left\vert#1\right\vert}
\newcommand{\set}[1]{\left\{#1\right\}}

\newcommand{\cardinality}[1]{\abs{#1}}

\newcommand{\diam}{\mathop{\mathrm{diam}}\nolimits}
\newcommand{\dist}{\mathop{\mathrm{dist}}\nolimits}

\newcommand{\norm}[1]{\left\Vert#1\right\Vert}

\newcommand{\closedball}[1]{B_{#1}}
\newcommand{\openball}[1]{B^O_{#1}}
\newcommand{\sphere}[1]{S_{#1}}

\newcommand{\Free}{{\mathcal F}}
\newcommand{\Lip}{{\mathrm{Lip}}_0}
\newcommand{\F}[1]{\mathcal{F}(#1)}
\newcommand{\conv}{\mathop\mathrm{conv}}
\newcommand{\ext}[1]{\mathrm{ext}\left(#1\right)}
\newcommand{\sep}{\mathrm{sep}}

\theoremstyle{plain}
\newtheorem{thm}{Theorem}
\newtheorem{cor}[thm]{Corollary}
\newtheorem{lem}[thm]{Lemma}
\newtheorem{prop}[thm]{Proposition}

\theoremstyle{definition}

\newtheorem{rem}[thm]{Remark}

\begin{document}
 \title{Characterization of metric spaces whose free space is isometric to $\ell_1$}
\author{Aude Dalet$^\dag$}
\author{Pedro L. Kaufmann$^\ddag$}
\author{Anton\'\i n Proch\'azka$^\dag$}
\address{$^\dag$ Universit\'e Franche-Comt\'e\\
Laboratoire de Math\'ematiques UMR 6623\\
16 route de Gray\\
25030 Besan\c con Cedex\\
France}
\address{$^\ddag$ Universidade Federal de S\~ao Paulo,  Instituto de Ci\^encia e Tecnologia, Campus S\~ao Jos\'e dos Campos - Parque Tecnol\'ogico, Avenida Doutor Altino Bondensan, 500, 12247-016 S\~ao Jos\'e dos Campos/SP, Brazil }
\email{aude.dalet@univ-fcomte.fr}
\email{plkaufmann@unifesp.br}
\email{antonin.prochazka@univ-fcomte.fr}

\begin{abstract}We characterize metric spaces whose Lipschitz free space is isometric to $\ell_1$.
In particular, the Lipschitz free space over an ultrametric space is not isometric to $\ell_1(\Gamma)$ for any set $\Gamma$. We give a lower bound for the Banach-Mazur distance in the finite case.
\end{abstract}

\maketitle
\section{Introduction}
An \emph{$\Real$-tree} $(T,d)$ is a metric space which is geodesic and satisfies the \emph{4-point condition}:
\[
\forall \, a,b,c,d \in T \quad d(a,b)+d(c,d)\leq \max \set{d(a,c)+d(b,d),d(b,c)+d(a,d)}.
\]
A space which satisfies just the 4-point condition is called \emph{$0$-hyperbolic}. Clearly, a subset of an $\Real$-tree is $0$-hyperbolic. The converse is also true \cite{Buneman,Evans}, so we will use terms ``$0$-hyperbolic'' and ``subset of an $\Real$-tree'' interchangeably. Moreover, for every $0$-hyperbolic $M$ there exists a unique (up to isometry) minimal $\Real$-tree which contains $M$, we will denote it $\conv(M)$. Thus one can define the Lebesgue measure $\lambda(M)$ of $M$ which is independent of any particular tree containing $M$.
We will say that $M$ is negligible if $\lambda(M)=0$.
A.~Godard \cite{G} has proved that a metric space $M$ is $0$-hyperbolic if and only if $\Free(M)$ is isometric to a subspace of some $L_1(\mu)$.
In this paper we are interested in metric spaces whose free space is isometric to (a subspace of) $\ell_1$. 
By the above, such spaces must be $0$-hyperbolic, and it is also easy to see that they must be negligible (if not the free space will contain $L_1$).   

So let $M$ be a separable negligible complete metric space which is a subset of an $\Real$-tree. One can ask two questions:
\begin{itemize}
 \item When is $\Free(M)$ isometric to $\ell_1$?
 \item When is $\Free(M)$ isometric to a subspace of $\ell_1$?
\end{itemize}
Concerning the first question, the results of A.~Godard point to the relevance of branching points of $\conv(M)$.
We recall that a point $b\in T$ is a branching point of a tree $T$ if $T\setminus \set{b}$ has at least three connected components.
A sufficient condition for $\Free(M)\equiv \ell_1$ is that $M$ contain all the branching points of $\conv(M)$ \cite[Corollary~3.4]{G}.
The main result of this paper (Theorem~\ref{t:main}) claims that this is also a necessary condition.
We give two different proofs -- one is based on properties of the extreme points of $\closedball{\Free(M)}$ and the other on properties of the extreme points of $\closedball{\Lip(M)}$ (Theorem~\ref{t:CloseExtremePoints}).

For certain finite $0$-hyperbolic spaces $M$ we have a third proof which also allows to compute a simple lower bound for the Banach-Mazur distance between $\Free(M)$ and $\ell_1^{\cardinality{M}-1}$ (Proposition~\ref{p:BM}). 

As far as the second question is concerned, it is obviously enough that $M$ be a subset of a metric space $N$ such that $\Free(N) \equiv \ell_1$. 
We will show that this is the case when $M$ is compact, $0$-hyperbolic and negligible (Proposition~\ref{p:compact}). 
We do not know whether one can drop the assumption of compactness in general.

This paper is an outgrowth of a shorter preprint in which we have shown that for any ultrametric space $M$, the free space $\Free(M)$ is never isometric to $\ell_1$ (Corollary~\ref{c:ultrametric}) answering a question posed by M. C\'uth and M. Doucha in a draft of~\cite{CD}. In the meantime, this question has been independetly answered in~\cite{CD}.

\section{Preliminaries}


As usual, for a metric space $M$ with a distinguished point $0 \in M$, the \emph{Lipschitz-free space} $\Free(M)$ is the norm-closed linear span of $\set{\delta_x: x \in M}$ in the space $\Lip(M)^*$, where the Banach space $\Lip(M)=\set{f \in \Real^M: f \mbox{ Lipschitz}, f(0)=0}$ is equipped with the norm $\displaystyle\norm{f}_L:=\sup\set{\frac{f(x)-f(y)}{d(x,y)}:x \neq y}$. It is well known that $\Free(M)^*=\Lip(M)$ isometrically. 
More about the very interesting class of Lipschitz-free spaces can be found in \cite{GodefroyKalton}.

\medskip

To prove a Lispchitz-free space is not isometric to $\ell_{1}$, we will exhibit two extreme points of its unit ball at distance less than one. 
For this purpose we will use the notion of {\it peaking function at $(x,y), x\neq y$}, which is a function $f\in \Lip(M)$ such that $\frac{f(x)-f(y)}{d(x,y)}=1$ and for every open set $U$ of $\{(x,y)\in M\times M, x\neq y\}$ containing $(x,y)$ and $(y,x)$, 
there exists $\delta >0$ with 
$	(z,t)\notin U \Rightarrow \frac{|f(z)-f(t)|}{d(z,t)}\leq 1-\delta.	$
%
%
This definition is equivalent to: $\frac{f(x)-f(y)}{d(x,y)}=1$ and if $(u_n)_{n\in \mathbb{N}}, (u_n)_{n\in \mathbb{N}}\subset M$, then 
\[	\lim\limits_{n\rightarrow +\infty}\frac{f(u_n)-f(v_n)}{d(u_n,v_n)}=1\Rightarrow \lim\limits_{n\rightarrow +\infty} u_n=x \textrm{\  and }	\lim\limits_{n\rightarrow +\infty} v_n=y.\]

Moreover in \cite[Proposition~2.4.2]{W}, the following is proved:
\begin{prop}\label{p:peak,extr}
Let $(M,d)$ be a complete metric space and $x\neq y$ in $M$. If there is a function $f \in \Lip(M)$ peaking at $(x,y)$, then $\frac{\delta_x-\delta_y}{d(x,y)}$ is an extreme point of the unit ball of $\Lip(M)^*$. In particular, it is an extreme point of the unit ball of $\F{M}$.
\end{prop}

\medskip

Given an $\mathbb{R}$-tree $(T,d)$ and $x,y\in T$, the {\it segment} $[x,y]$ is defined as the range of the unique isometry $\phi_{x,y}$ from $[0,d(x,y)]\subset\mathbb{R}$ into $T$ which maps $0$ to $x$ and $d(x,y)$ to $y$.

We recall that for every $0$-hyperbolic space $M$, there exists an $\Real$-tree $T$ such that $M \subset T$. The set $\bigcup \set{[x,y]:x,y \in M} \subset T$ is then also an $\Real$-tree. It is clearly a minimal $\mathbb{R}$-tree containing~$M$; it is unique up to an isometry and will be denoted conv$(M)$.
Simple examples show that $\conv(M)$ does not have to be complete when $M$ is. This does not suppose any difficulty in what follows.

A point $b\in T$ is said to be a {\it branching point} if there are three distinct points $x,y,z\in T\backslash \{b\}$ with $[x,b]\cap[y,b]=[x,b]\cap [z,b]=[y,b]\cap[z,b]=\{b\}$.
We say that the branching point $b$ is witnessed by $x,y,z$.
The set of all branching points of $T$ is denoted $Br(T)$.
If $M$ is $0$-hyperbolic, the set of all branching points of $\conv(M)$ is denoted $Br(M)$.

A subset $A$ of $T$ is {\it measurable} if $\phi_{x,y}^{-1}(A)$ is Lebesgue-measurable, for every $x$ and $y$ in $T$.
For a segment $S=[x,y]$ in $T$ and $A$ measurable, we denote $\lambda_{S}(A):=\lambda(\phi_{x,y}^{-1}(A))$, with $\lambda$ the Lebesgue measure on $\mathbb{R}$. 
Let $\mathcal{R}$ be the set of subsets of $T$ that can be written as a finite union of disjoint segments. For $R=\displaystyle\bigcup_{k=1}^{r}S_{k}\in \mathcal{R}$, define $\lambda_{R}(A):=\sum\limits_{k=1}^{r}\lambda_{S_{k}}(A)$ and finally, set 
$\lambda_{T}(A):=\displaystyle\sup_{R\in \mathcal{R}}\lambda_{R}(A)$.
If $M$ is $0$-hyperbolic, we put simply $\lambda(M):=\lambda_{\conv(M)}(M)$.

Given two points $x$ and $y$ in $T$, we will denote $\pi_{xy}:T \to [x,y]$ the metric projection onto the segment $[x,y]$. 
It is well known and easily seen that $\pi_{xy}$ is non-expansive (see~\cite{Bacak,BH}).

\medskip

Finally, we recall that a metric space $(M,d)$ is \emph{ultrametric} if $d(x,y)\leq \max\set{d(x,z),d(y,z)}$ for any $x,y,z \in M$.

\section{Isometries with $\ell_1$}
Let us start by characterizing precisely when there exists a function peaking at $(x,y)$ for points $x,y \in M \subset T$.

\begin{prop}\label{p:peaking}
Let $(M,d)$ be a complete subset of an $\mathbb{R}$-tree and $x,y\in M$, $x\neq y$.
The following assertions are equivalent
\begin{itemize}
\item[(i)] There is $f \in \Lip(M)$ peaking at $(x,y)$.
\item[(ii)] $M \cap [x,y]=\set{x,y}$ and 
for every $p\in \{x,y\}$, 
	\begin{eqnarray}\label{e:tang}
    \liminf\limits_{u,v\rightarrow p}
\frac{d(\pi_{xy}(u),u)+d(\pi_{xy}(v),v)}{d(\pi_{xy}(u),\pi_{xy}(v))}>0, \, (\mbox{with the convention that } \frac{\alpha}{0}=+\infty).
    \end{eqnarray}
\item[(iii)] $M \cap [x,y]=\set{x,y}$ and 
for every $p\in \{x,y\}$, 
	\begin{eqnarray}\label{e:tang2}
    \liminf\limits_{u\rightarrow p}
\frac{d(\pi_{xy}(u),u)}{d(\pi_{xy}(u),p)}>0, \, (\mbox{with the convention that } \frac{\alpha}{0}=+\infty).
    \end{eqnarray}
\end{itemize}
\end{prop}
\begin{proof}
(ii) $\Rightarrow$ (i) Let us first suppose that $x,y$ satisfy \eqref{e:tang} and $[x,y] \cap M=\set{x,y}$.
For any $u \in M$ we define $f(u)=d(y,\pi_{xy}(u))$.
Then $\displaystyle\frac{f(x)-f(y)}{d(x,y)}=1$ and $\norm{f}_L=1$. 
Consider $(x_n)_{n\in \mathbb{N}}, (y_n)_{n\in \mathbb{N}}\subset M$ such that $\displaystyle\lim\limits_{n\rightarrow +\infty}\frac{f(x_n)-f(y_n)}{d(x_n,y_n)}=1$. 
We thus have for $n$ large enough 
\begin{eqnarray}\label{e:order}
d(y,\pi_{xy}(x_n))=f(x_n)>f(y_n)=d(y,\pi_{xy}(y_n)).
\end{eqnarray}
It follows 
\[
1=\lim\limits_{n\rightarrow+\infty}\frac{f(x_n)-f(y_n)}{d(x_n,y_n)}
				=\lim\limits_{n\rightarrow+\infty}
					\frac{d(\pi_{xy}(x_n),\pi_{xy}(y_{n}))}
						{d(x_n,\pi_{xy}(x_{n}))+d(\pi_{xy}(x_n),\pi_{xy}(y_{n}))+d(\pi_{xy}(y_{n}),y_{n})}
\]
and in particular 
\begin{equation}\label{e:IfTangent}
\lim_{n \to \infty} \frac{d(x_n,\pi_{xy}(x_n))+d(\pi_{xy}(y_n),y_n)}{d(\pi_{xy}(x_n),\pi_{xy}(y_n))}=0.
\end{equation}
Since $\lim\limits_{n\rightarrow +\infty}d(x_{n},\pi_{xy}(x_{n}))
=\lim\limits_{n\rightarrow +\infty}d(y_{n},\pi_{xy}(y_{n}))=0$, the sets of cluster points of the sequences
$((\pi_{xy}(x_{n}),\pi_{xy}(y_{n})))_{n\in \mathbb{N}}\subset[x,y]^2$ and $((x_{n},y_{n}))_{n\in \mathbb{N}}\subset M^2$ coincide.
By compactness of $[x,y]^2$ there exists such a cluster point $(u,v)\in [x,y]^2$.
Since the space $M$ is complete, $(u,v)\in M^2$, and therefore
	$(u,v)\in\{(y,x),(x,x),(y,y),(x,y)\}. $
Clearly, \eqref{e:order} implies $(u,v)\neq (y,x)$, and \eqref{e:tang} together with \eqref{e:IfTangent} imply that $(u,v) \neq (x,x)$ and $(u,v) \neq (y,y)$. 
We thus get that $(x_n)$ converges to $x$ and $(y_n)$ converges to $y$ which proves that $f$ is peaking at $(x,y)$. 

(i) $\Rightarrow$ (iii)
If there is $z \in M \cap (x,y)$, then $\frac{\delta_x-\delta_y}{d(x,y)}$ is a convex combination of $\frac{\delta_x-\delta_z}{d(x,z)}$ and $\frac{\delta_z-\delta_y}{d(z,y)}$ so by Proposition~\ref{p:peak,extr}, there cannot be a peaking function at $(x,y)$.

Next assume that $[x,y] \cap M=\set{x,y}$ but there is a sequence $(u_{n})_{n\in \mathbb{N}}\subset M$ converging to $x$ and 
\[
	\lim\limits_{n\rightarrow +\infty}\frac{d(\pi_{x,y}(u_{n}),u_{n})}{d(\pi_{x,y}(u_{n}),x)}=0.
\]
Let $f\in \sphere{\Lip(M)}$ be such that $\frac{f(x)-f(y)}{d(x,y)}=1$. Let $\widetilde{f}$ be a 1-Lipschitz extension of $f$ to $[x,y]$. 
Then 
	\begin{align*}
	|f(x)-f(u_{n})|&\geq |f(x)-\widetilde{f}(\pi_{xy}(u_{n}))|-|\widetilde{f}(\pi_{xy}(u_{n}))-f(u_{n})|\\
		&=d(x,\pi_{xy}(u_{n}))-|\widetilde{f}(\pi_{xy}(u_{n}))-f(u_{n})|\\
		&\geq d(x,\pi_{xy}(u_{n}))-d(\pi_{xy}(u_{n}),u_{n})\\
		&\geq d(x,u_{n})-2d(\pi_{xy}(u_{n}),u_{n})
	\end{align*}
It follows that 
\[
\lim\limits_{n\rightarrow+\infty}\frac{|f(x)-f(u_{n})|}{d(x,u_{n})}=1.
\]
and $f$ is not peaking at $(x,y)$.

(iii) $\Rightarrow$ (ii)
Finally, since 
\[
\frac{d(u,\pi_{xy}(u))+d(v,\pi_{xy}(v))}{d(\pi_{xy}(u),\pi_{xy}(v))} \geq \min\set{\frac{d(\pi_{xy}(u),u)}{d(\pi_{xy}(u),p)},\frac{d(\pi_{xy}(v),v)}{d(\pi_{xy}(v),p)}}
\]
we get 
\[
\liminf\limits_{u \to p}\frac{d(\pi_{xy}(u),u)}{d(\pi_{xy}(u),p)}=0
\]
if the liminf in \eqref{e:tang} is $0$ for some $p \in \set{x,y}$.
\end{proof}

For the dual version of the proof we will need the following simple lemma which is valid in any metric space (see also \cite{Farmer}). 
\begin{lem}\label{l:ext}
Let $(M,d)$ be any metric space and suppose that $0 \in A \subset M$. If $f \in \ext{\closedball{\Lip(A)}}$, then $f_S,f_I \in \ext{\closedball{\Lip(M)}}$ where
\[
f_S(x):=\sup_{z\in A} f(z)-d(z,x) \quad \mbox{and} \quad f_I(x):=\inf_{z \in A} f(z)+d(z,x)
\]
for $x \in M$.
\end{lem}
Note that $f_S$ resp. $f_I$ above are the smallest resp. the largest 1-Lipschitz extensions of $f$ (which basically gives the proof).

\begin{proof}
Let us give a proof for $f_S$. The proof for $f_I$ is similar.
Clearly $f_S(x)=f(x)$ for $x \in A$ and $f_S$ is 1-Lipschitz as a supremum of $1$-Lipschitz functions.
Let ${f_S}=\frac{p+q}{2}$, $p,q \in \closedball{\Lip{M}}$.
If $x \in A$, then $p(x)=q(x)=f(x)$ as $f \in \ext{\closedball{\Lip{A}}}$.
If $x \in M \setminus A$, then $\forall \, z \in A$:
\[
f(z)-p(x)=p(z)-p(x)\leq d(z,x).
\]
Thus
\[
{f_S}(x)=\sup_{z \in A} f(z)-d(z,x) \leq p(x)
\]
By the same argument ${f_S}(x)\leq q(x)$. So $f_S(x)=p(x)=q(x)$ for all $x \in M$.
\end{proof}

\begin{thm}\label{t:CloseExtremePoints}
 Let $M$ be a complete subset of an $\Real$-tree. If there is $b \in Br(M)\setminus M$ then
 \begin{itemize}
  \item[a)] there exist $\mu \neq  \nu \in \ext{\closedball{\Free(M)}}$ such that $\norm{\mu-\nu}<2$.
  \item[b)] there exist $f \neq g \in \ext{\closedball{\Lip(M)}}$ such that $\norm{f-g}_L<2$.
 \end{itemize}
\end{thm}

Since the Lipschitz-free space over the completion of $M$ equals the Lipschitz-free space of $M$, this completness hypothesis is not restrictive.

\begin{proof}
{\bf a)} Let the points $x',y',z' \in M$ witness that $b \in Br(M)$. For $p' \in \set{x',y',z'}$ we denote $M_{p'}=\set{w \in M: \pi_{bp'}(w) \in ]b,p']}$. 
Then $M_{p'}$ is closed in $M$ as $\pi_{bp'}$ is continuous and $b$ is isolated from $M$.
Notice that $p \in M_{p'}$ satisfies \eqref{e:tang2} if there is $\alpha>0$ such that  $\displaystyle d(w,\pi_{bp}(w))\geq \alpha d(p,\pi_{bp}(w))$ for all $w \in M_{p'}$. 
We will show that for every $0<\alpha<1$ such a point $p$ exists.
Indeed let $\frac{1-\alpha}{1+\alpha}=:\beta>0$ and set $f(w):=d(b,w)$. 
Then the Ekeland's variational principle \cite{E} ensures the existence of a point $p \in M_{p'}$ such that $f(p)\leq f(w)+\beta d(p,w)$ for all $w \in M_{p'}$. It follows that

$$
\begin{array}{crcl}
& d(b,\pi_{bp}(w))+d(\pi_{bp}(w),p) &\leq & d(b,\pi_{bp}(w))+d(\pi_{bp}(w),w)+\beta d(p,w)\\
\Longrightarrow \,& d(\pi_{bp}(w),p) &\leq & d(\pi_{bp}(w),w)+\beta (d(\pi_{bp}(w),w)+d(\pi_{bp}(w),p))\\
\Longrightarrow \,& \frac{1-\beta}{1+\beta} d(p,\pi_{bp}(w)) &\leq& d(w,\pi_{bp}(w)).
\end{array}
$$

Thus, we see that we can find $x,y,z \in M$ such that (iii) in Proposition~\ref{p:peaking} is satisfied for the segments $[p,q]$ where $p\neq q \in \set{x,y,z}$. 
Proposition \ref{p:peak,extr} then yields that $\frac{\delta_{p}-\delta_{q}}{d(p,q)}$ is an extreme point of the unit ball of $\F{M}$.
Assuming, as we may, that $d(x,z)\leq d(z,y)\leq d(x,y)$, we obtain
\begin{align*}
			\left\|\frac{\delta_{x}-\delta_{y}}{d(x,y)}-\frac{\delta_{z}-\delta_{y}}{d(y,z)}\right\|_{\F{M}}
			&=\left\|\frac{1}{d(x,y)}\left[(\delta_{x}-\delta_{z})+(\delta_{z}-\delta_{y})\right]
			-\frac{\delta_{z}-\delta_{y}}{d(y,z)}\right\|_{\F{M}}\\
			&=\left\|\left[\frac{1}{d(x,y)}-\frac{1}{d(y,z)}\right](\delta_{z}-\delta_{y})
			+ \frac{\delta_{x}-\delta_{z}}{d(x,y)} \right\|_{\F{M}}\\
			&\leq d(z,y)\left[\frac{1}{d(y,z)}-\frac{1}{d(x,y)}\right]
			+\frac{d(x,z)}{d(x,y)}\\
			&=1+\frac{d(x,z)-d(z,y)}{d(x,y)}\leq 1.
\end{align*}
In conclusion, $\mu:=\frac{\delta_{x}-\delta_{y}}{d(x,y)}$ and $\nu:=\frac{\delta_{z}-\delta_{y}}{d(y,z)}$ are two extreme points of the unit ball of $\F{M}$ at distance less than or equal to 1.

{\bf b)} We denote $\delta:=\inf\set{d(w,b):w \in M}$. 
Let $x,y,z$ be 3 points witnessing the fact that $b$ is a branching point. 
Two pointed metric spaces which differ only by the choice of the base point have isometric free spaces. 
This trivial observation allows us to assume that $x=0$ and that, for a fixed $0<\varepsilon<1$, we have $d(b,z)<(1+\varepsilon)\delta$.
Let $M_z=\set{w \in M: \pi_{zb}(w) \in (b,z]}$. 
Let us consider the closed nonempty set $F=\set{w \in M_z: d(b,z)\leq (1+\varepsilon)\delta}$. 
Given $0<\alpha<1$ and using the Ekeland's variational principle as above,  we may assume that $z$ satisfies $\displaystyle d(w,\pi_{zb}(w))\geq \alpha d(z,\pi_{zb}(w))$ for all $w \in F$. 
Clearly $\displaystyle d(w,\pi_{zb}(w))\geq \alpha d(z,\pi_{zb}(w))$ for all $w \in M_z\setminus F$.

We define $f(\cdot):=d(0,\cdot)$ on $M$ and then $g_2(\cdot):=d(0,\cdot)$ on $M\setminus M_z$, $g_1:=(g_2)_S$ on $(M \setminus M_z) \cup \set{z}$ and finally $g:=(g_1)_I$ on $M$.
Both $f,g \in \ext{\closedball{\Lip(M)}}$ by  Lemma~\ref{l:ext}.
The fact that $M$ is a subset of an $\Real$-tree helps to write $g$ explicitely:
\[
g(w)=\begin{cases}
d(0,w),& w \in M\setminus M_z,\\
d(0,b)-d(b,z)+d(z,w),& w \in M_z. 
\end{cases}
\]
It follows that
$f(w)-g(w)=0$ for $w \in M \setminus M_z$ and $f(w)-g(w)=2d(b,\pi_{zb}(w))$ otherwise.
We have
\[
\begin{split}
\norm{f-g}_L&=\max\set{\sup_{w_1 \in M_z, w_2 \notin M_z}\frac{2d(b,\pi_{zb}(w_1))}{d(w_1,w_2)},\sup_{w_1,w_2 \in M_z}\frac{2\abs{d(w_1,\pi_{zb}(w_1))-d(w_2,\pi_{zb}(w_2))}}{d(w_1,w_2)}}\\
&\leq \max\set{\frac{2(1+\varepsilon)\delta}{2\delta},\frac{2}{1+\alpha}}<2
\end{split}
\]
\end{proof}

\begin{thm}\label{t:main}
Let $(M,d)$ be a complete metric space. The Lipschitz-free space over $M$ is isometric to $\ell_1(\Gamma)$ if and only if $M$ is of density $\cardinality{\Gamma}$ and is negligible subset of an $\Real$-tree $T$ which contains all the branching points of $T$.
\end{thm}

\begin{proof}
The sufficiency follows from \cite[Theorem~3.2]{G}.
Conversely, let us assume that $\Free(M)\equiv \ell_1(\Gamma)$. 
Then $M$ is of density $\cardinality{\Gamma}$ and it must be $0$-hyperbolic by \cite[Theorem~4.2]{G}. 
In this case $T=\conv(M)$.
If $\lambda_{T}(M)>0$, there is a set $A \subset [0,1]$ of positive measure such that $A$ embeds isometrically into $M$. 
Then $L_1 \simeq \Free(A) \subset \Free(M)\equiv \ell_1(\Gamma)$ which is absurd.
Since the extreme points of the ball (resp. dual ball) and their distances are preserved by bijective isometries we get by Theorem~\ref{t:main} a) (resp. b)) that $Br(M) \subset M$.
\end{proof}

\begin{cor}\label{c:ultrametric}
Let $M$ be an ultrametric space of cardinality at least $3$. Then $\Free(M)$ is not isometric to $\ell_1(\Gamma)$ for any $\Gamma$.
\end{cor}

\begin{proof}
The completion of $M$ stays clearly ultrametric. Thus it can be isometrically embedded into an $\Real$-tree \cite{Buneman}. However ultrametric spaces do not contain the interior of any segment, much less branching points.
\end{proof}

\section{Isometries with subspaces of $\ell_1$}

We can now ask whether $\F{M}$ is isometric to a subspace of $\ell_1$?
We are going to answer this question in the affirmative in the case when $M$ is supposed to be compact. 

\begin{lem}
Let $M$ be a compact subset an $\Real$-tree such that $\lambda(M)=0$. Then $\lambda_{\conv(M)}(\overline{Br(M)})=0$ where the closure is taken in $\conv(M)$. 
\end{lem}
\begin{proof}
Clearly $\lambda_{\conv(M)}(\overline{Br(M)}\cap M)=0$. 
Assume that $\lambda_{\conv(M)}(\overline{Br(M)} \setminus M)>0$. Then $\overline{Br(M)}\setminus M$ is uncountable. Hence there is some $\delta>0$ such that $\overline{Br(M)} \cap \set{x \in T: \dist(x,M)\geq \delta}$ is uncountable and thus the set $Br(M) \cap \set{x \in T: \dist(x,M)\geq \frac{\delta}{2}}$ is infinite. We conclude that there is an infinite $\delta$-separated family in $M$. This is absurd as $M$ was supposed to be compact. 
\end{proof}

\begin{prop}\label{p:compact}
Let $M$ be a compact subset of an $\Real$-tree such that $\lambda(M)=0$. Then $\Free(M)$ is isometric to a subspace of $\ell_1$. 
\end{prop}
\begin{proof}
Since $M$ is compact, $\conv(M)$ is compact and thus separable. Indeed, the mapping $\Phi:M\times M\times [0,1] \to \conv(M)$ defined by $\Phi(x,y,t):=\phi_{xy}(td(x,y))$ is continuous by \cite[Theorem II.4.1]{BH}. 
Now \[\Free(M) \subseteq \Free(Br(M) \cup M) \equiv \ell_1\] by \cite[Corollary 3.4]{G} as $\lambda_{\conv(M)}(\overline{Br(M) \cup M})=0$ by the previous lemma.
\end{proof}

We do not know if the above proposition is valid when $M$ is supposed to be proper.

\section{Banach-Mazur distance to $\ell_1^n$}
In the case of finite subsets of $\Real$-trees we get the following quantitative result.
\begin{prop}\label{p:BM}
Let $M=\set{x_0,x_1,\ldots,x_n}$, $n\geq 2$, be a subset of a $\Real$-tree. Let $x_0=0$ be the distinguished point. Let us suppose that 
\[0<\sep(M):= \frac12\inf\set{d(x,y)+d(x,z)-d(y,z):x, y, z \in M \mbox{ distinct}}.\] 
Then
\[
d_{BM}(\Free(M),\ell_1^n)>\left(1-\frac{\sep(M)}{4\diam(M)}\right)^{-1}.
\]
\end{prop}

The condition $\sep(M)>0$ implies immediately that for each $x\neq y \in M$ we have $[x,y] \cap M=\set{x,y}$.
For the proof we will need the following lemmas. The first one is inspired by
\cite[Lemma~2.3]{BV}.

\begin{lem}
Let $X$ be a Banach space. Let $C=\bigcap_{i=1}^n x_i^{*-1}(-\infty,1)$ where $x_i^* \in X^*$. Let $A \subset X\setminus C$ have the following property: for every $x\neq y\in A$, we have $\displaystyle\frac{x+y}{2} \in C$. Then the cardinality $\cardinality{A}$ of $A$ is at most $n$.
\end{lem}

\begin{proof}
For $x \in A$ let $\varphi(x):=i$ for some $i\in \{1,\dots,n\}$ such that $x^*_i(x)\geq 1$. Since $\displaystyle 1> x^*_{\varphi(x)}\left(\frac{x+y}{2}\right)$ it follows that $x^*_{\varphi(x)}(y)<1$ for every $y \in A$, $y \neq x$. Thus $\varphi$ is injective and the claim follows.
\end{proof}

\begin{lem}\label{l:BMellinfty}
Let $f_1,\ldots,f_{2n+1} \in \sphere{Y}$ such that $\norm{\frac{f_i+f_j}2}\leq 1-\varepsilon$ for some $\varepsilon>0$ and all $1\leq i\neq j\leq 2n+1$. Then $d_{BM}(Y,\ell_\infty^n)>(1-\varepsilon)^{-1}$.
\end{lem}
\begin{proof}
Let $T:Y \to \ell_\infty^n$ such that $\norm{f} \leq \norm{Tf}_\infty \leq (1+\varepsilon)\norm{f}$.
Then $\norm{Tf_i}\geq 1$, $\norm{\frac{Tf_i+Tf_j}2}<1$ which is in contradiction with the previous lemma as $\openball{\ell_\infty^n}$ is the intersection of $2n$ halfspaces.
\end{proof}

\begin{proof}[Proof of Proposition~\ref{p:BM}]
Given $0 \leq i\neq j \leq n$, we will denote $\pi_{ij}:=\pi_{x_ix_j}$ the metric projection onto $[x_i,y_j]$. 
Further we define the function $f_{ij}:M \to \Real$ as $f_{ij}(z):=d(x_j,\pi_{ij}(z))$ for $z \in M$. Observe that since $\sep(M)>0$, this is the function peaking at $(x_i,x_j)$ from the proof of Proposition~\ref{p:peaking}.
It is clear that $\abs{\frac{f_{ij}(x)-f_{ij}(y)}{d(x,y)}}=1$ if and only if $\set{x,y}=\set{x_i,x_j}$. 
We further have that 
\[
\abs{\frac{f_{ij}(x)-f_{ij}(y)}{d(x,y)}}\leq \frac{d(x,y)-\sep(M)}{d(x,y)}\leq 1-\frac{\sep(M)}{\diam M}
\] 
for any other couple $x\neq y \in M$. Hence $\norm{\frac{f_{ij}+f_{kl}}{2}}_L \leq 1-\frac{\sep(M)}{2\diam M}$ for each $(i,j)\neq (k,l)$. Since $n \geq 2$, we have that $(n+1)n\geq 2n+1$ and the result follows by Lemma~\ref{l:BMellinfty}.
\end{proof}

\begin{rem}
Note that the lower bound given in Proposition~\ref{p:BM} is not optimal. This can be seen when $M=\set{0,x_1,x_2}$ is equilateral. We also don't know if this result extends to infinite subsets of $\Real$-trees.
\end{rem}


\begin{thebibliography}{A}
\bibitem{Bacak}{M. Ba\v c\'ak. Convex analysis and optimization in Hadamard spaces, De Gruyter, 2014.}
\bibitem{BV}{J. Borwein, J. Vanderwerff. {\it Constructible convex sets},  Set-Valued Anal. 12 (2004), no. 1, 61-77.}


\bibitem{BH}{M.R. Bridson and A. Haefliger, Metric spaces of non-positive curvature, Springer, 1999.}

\bibitem{Buneman}{P. Buneman. {\it A note on the metric properties of trees}, J. Combinatorial
Theory Ser. B. 17 (1974) 48-50.}

\bibitem{CD}{M. C\'uth \and M. Doucha. {\it Lipschitz-free spaces over ultrametric spaces},  Mediterr. J. Math. (2015) DOI 10.1007/s00009-015-0566-7}

\bibitem{E}{I. Ekeland. {\it Nonconvex minimization problems}, Bull. Amer. Math. Soc. 1 (1979), no. 3, 443-474}

\bibitem{Evans}{S. N. Evans. 
Probability and Real Trees, LNM 1920, Springer, 2008.}

\bibitem{Farmer}{J.D. Farmer. {\it Extreme points of the unit ball of the space of Lipschitz functions}. Proc. Amer. Math. Soc. 121 (1994), no 3, 807-813.
}
\bibitem {G}{ A. Godard, {\it Tree metrics and their Lipschitz-free spaces}, Proc. Amer. Math. Soc.  138 (2010), no. 12, 4311-4320. }

\bibitem{GodefroyKalton}{G. Godefroy and N.J. Kalton. {\it Lipschitz-free Banach spaces}. Studia Math. 159 (2003), no. 1, 121-141. }

\bibitem{W}{N. Weaver. {\it Lipschitz algebras}. World Scientific Publishing Co. Inc., River Edge, NJ, 1999.}
\end{thebibliography}
\end{document}